\documentclass[12pt,draft,twoside]{article}

\usepackage[leqno]{amsmath}
\usepackage{latexsym,amsthm,amsfonts,amssymb,hyperref}
\usepackage[margin=2.54cm]{geometry}

\newtheorem{prop}{Proposition}
\newtheorem{lemm}[prop]{Lemma}
\theoremstyle{definition}
\newtheorem{remk}[prop]{Remark}
\newtheorem{algo}[prop]{Algorithm}

\newcommand{\prpref}[2][Proposition~]{#1\ref{prp:#2}}
\newcommand{\lemref}[2][Lemma~]{#1\ref{lem:#2}}
\newcommand{\rmkref}[2][Remark~]{#1\ref{rmk:#2}}
\newcommand{\secref}[2][Section~]{#1\ref{sec:#2}}

\newcommand{\clC}[1]{\mathbf{C}(#1)} \newcommand{\clS}{\mathbf{S}}
\newcommand{\clT}{\mathbf{T}}
\newcommand{\clB}{\mathbf{B}}
\newcommand{\clG}[1]{\mathbf{G}(#1)}
\newcommand{\clH}[1]{\mathbf{H}(#1)}
\newcommand{\Tor}[4][R]{\operatorname{Tor}^{#1}_{#2}(#3,#4)}
\newcommand{\Hom}[3][R]{\operatorname{Hom}_{#1}(#2,#3)}
\newcommand{\Ext}[4][R]{\operatorname{Ext}_{#1}^{#2}(#3,#4)}
\newcommand{\poly}[2][k]{#1[#2]}
\newcommand{\rnk}[2][k]{\operatorname{rank}_{#1}#2}
\newcommand{\pd}[2][R]{\operatorname{proj.\!dim}_{#1}#2}
\newcommand{\lra}{\longrightarrow}
\newcommand{\dptR}{\operatorname{depth}R}
\newcommand{\dimR}{\operatorname{dim}R}
\newcommand{\dpt}[2][R]{\operatorname{depth}_{#1}#2}
\newcommand{\mapdef}[4][\rightarrow]{\nobreak{#2\colon #3 #1 #4}}
\newcommand{\qtext}[1]{\quad\text{#1}\quad}
\newcommand{\qqtext}[1]{\qquad\text{#1}\qquad}
\newcommand{\qand}{\qtext{and}}
\newcommand{\qqand}{\qqtext{and}}
\renewcommand{\b}{\beta}
\newcommand{\n}{\mathfrak{n}}
\newcommand{\QQ}{\mathbb{Q}}

\begin{document}

\title{Local rings of embedding codepth 3:\\ a classification
  algorithm}

\author{Lars Winther Christensen\thanks{Part of this work was done
    while the authors visited MSRI during the Commutative Algebra
    program in spring 2013. LWC was partly
    supported by NSA grant H98230-11-0214.}\\
  Texas Tech University,\\
  Lubbock, TX 79409, U.S.A.\\
  lars.w.christensen@ttu.edu \and
  Oana Veliche\\
  Northeastern University,\\
  Boston, MA~02115, U.S.A.\\
  o.veliche@neu.edu}

\date{25 September 2014}

\maketitle

\begin{abstract}
  Let $I$ be an ideal of a regular local ring $Q$ with residue field
  $k$. The length of the minimal free resolution of $R=Q/I$ is called
  the codepth of $R$. If it is at most $3$, then the resolution
  carries a structure of a differential graded algebra, and the
  induced algebra structure on $\Tor[Q]{\ast}{R}{k}$ provides for a
  classification of such local rings.

  We describe the \emph{Macaulay\,2} package \emph{CodepthThree} that
  implements an algorithm for classifying a local ring as above by
  computation of a few cohomological invariants.
\end{abstract}

\thispagestyle{empty}

\section{Introduction and notation}

Let $R$ be a commutative noetherian local ring with residue field
$k$. Assume that $R$ has the form $Q/I$ where $Q$ is a regular local
ring with maximal ideal $\n$ and $I \subseteq \n^2$. The embedding
dimension of $R$ (and of $Q$) is denoted $e$. Let
\begin{equation*}
  F = 0 \lra F_c \lra \cdots \lra F_1 \lra F_0 \lra 0
\end{equation*}
be a minimal free resolution of $R$ over $Q$. Set $d = \dptR$; the
length $c$ of the resolution $F$ is by the Auslander--Buchsbaum
formula
\begin{equation*}
  c = \pd[Q]{R} = \dpt[]{Q} - \dpt[Q]{R} = e-d,
\end{equation*}
and one refers to this invariant as the \emph{codepth} of $R$. In the
following we assume that $c$ is at most $3$. By a theorem of Buchsbaum
and Eisenbud \cite[3.4.3]{bruher} the resolution $F$ carries a
differential graded algebra structure, which induces a unique
graded-commutative algebra structure on $A = \Tor[Q]{\ast}{R}{k}$. The
possible structures were identified by Weyman~\cite{JWm89} and by
Avramov, Kustin, and Miller \cite{AKM-88}. According to the
multiplicative structure on $A$, the ring $R$ belongs to exactly one
of the classes designated $\clB$, $\clC{c}$, $\clG{r}$, $\clH{p,q}$,
$\clS$, and $\clT$. Here the parameters $p$, $q$, and $r$ are given by
\begin{equation*}
  p = \rnk{(A_1\cdot A_1)}, \quad 
  q = \rnk{(A_1\cdot A_2)}, \qand 
  r = \rnk{(\mapdef{\delta}{A_2}{\Hom[k]{A_1}{A_3}})},
\end{equation*}
where $\delta$ is the canonical map. See \cite{LLA12,AKM-88,JWm89} for
further background and details.

When, in the following, we talk about classification of a local ring
$R$, we mean the classification according to the multiplicative
structure on $A$. To describe the classification algorithm, we need a
few more invariants of $R$. Set
\begin{equation*}
  l = \rnk[Q]{F_1} - 1 \qqand n = \rnk[Q]{F_c};
\end{equation*}
the latter invariant is called the \emph{type} of $R$.  The
\emph{Cohen--Macaulay defect} of $R$ is $h = \dimR - d$.  The Betti
numbers $\b_i$ and the Bass numbers $\mu_i$ record ranks of cohomology
groups,
\begin{equation*}
  \b_i  = \b_i^R(k) = \rnk{\Ext{i}{k}{k}} \qqand   \mu_i =
  \mu_i(R)  = \rnk{\Ext{i}{k}{R}}.
\end{equation*}
The generating functions $\sum_{i=0}^{\infty}\b_it^i$ and
$\sum_{i=0}^{\infty}\mu_it^i$ are called the \emph{Poincar\'e series}
and the \emph{Bass series} of $R$.

\section{The algorithm}

For a local ring of codepth $c \le 3$, the class together with the
invariants $e$, $c$, $l$, and $n$ completely determine the Poincar\'e
series and the Bass series of $R$; see \cite{LLA12}. Conversely, one
can determine the class of $R$ based on $e$, $c$, $l$, $n$, and a few
Betti and Bass numbers; in the following we describe how.

\begin{lemm}
  \label{lem:pqr}
  For a local ring $R$ of codepth $3$ the invariants $p$, $q$, and $r$
  are determined by $e$, $l$, $n$, $\b_2$, $\b_3$, $\b_4$, and
  $\mu_{e-2}$ through the formulas
  \begin{align*}
    p &= n + le + \b_2 - \b_3 + \textstyle\binom{e-1}{3},\\
    q &= (n-p)e +l\b_2 + \b_3  -\b_4 + \textstyle\binom{e-1}{4}, \ \text{ and}\\
    r&= l+n-\mu_{e-2}.
  \end{align*}
\end{lemm}

\begin{proof}
  The Poincar\'e series of $R$ has by \cite[2.1]{LLA12} the form
  \begin{equation}
    \sum_{i=0}^\infty \b_it^i 
    = \frac{(1+t)^{e-1}}{1-t-lt^2-(n-p)t^3+qt^4 + \cdots}\:,
  \end{equation}
  and expansion of the rational function yields the expressions for
  $p$ and $q$.

  One has $d = e-3$ and the Bass series of $R$ has, also by
  \cite[2.1]{LLA12}, the form
  \begin{equation}
    \sum_{i=0}^\infty \mu_it^i = t^d\frac{n + (l-r)t + \cdots}{1-t + \cdots}\;;
  \end{equation}
  expansion of the rational function now yields the expression for
  $r$.
\end{proof}

\begin{prop}
  \label{prp:pqr}
  A local ring $R$ of codepth $3$ can be classified based on the
  invariants $e$, $h$, $l$, $n$, $\b_2, \b_3,\b_4$, $\mu_{e-2}$, and
  $\mu_{e-1}$.
\end{prop}

\begin{proof}
  First recall that one has $h=0$ and $n=1$ if and only if $R$ is
  Gorenstein; see \cite[3.2.10]{bruher}. In this case $R$ is in class
  $\clC{3}$ if $l=2$ and otherwise in class $\clG{l+1}$.

  Assume now that $R$ is not Gorenstein. The invariants $p$, $q$, and
  $r$ can be computed from the formulas in \lemref{pqr}. It remains to
  determine the class, which can be done by case analysis. Recall from
  \cite[1.3 and 3.1]{LLA12} that one has
  \begin{equation*}
    \begin{array}{r|ccc}
      Class &  p & q & r\\
      \hline
      \clT & 3 &0 &0 \\
      \clB & 1 &1 &2 \\
      \clG{r}\ [r\ge 2]& 0 &1 &r \\
      \clH{p,q} &  p & q & q\\
    \end{array}
  \end{equation*}  
  In case $q \ge 2$ the ring $R$ is in class $\clH{p,q}$; for $q\le 1$
  the case analysis shifts to $p$.

  In case $p=0$ the distinction between the classes $\clG{r}$ and
  $\clH{0,q}$ is made by comparing $q$ and $r$; they are equal if and
  only if $R$ is in class $\clH{0,q}$.

  In case $p=1$ the distinction between the classes $\clB$ and
  $\clH{1,q}$ is made by comparing $q$ and $r$; they are equal if and
  only if $R$ is in class $\clH{1,q}$.

  In case $p=3$ the distinction between the classes $\clT$ and
  $\clH{3,q}$ is drawn by the invariant $\mu_{e-1}$. Recall the
  relation $d=e-3$; expansion of the expressions from
  \cite[2.1]{LLA12} yields $\mu_{e-1} = \mu_{e-2} + ln - 2$ if $R$ is
  in $\clT$ and $\mu_{e-1} = \mu_{e-2} + ln - 3$ if $R$ is in
  $\clH{3,q}$.

  In all other cases, i.e. $p=2$ or $p \ge 4$, the ring $R$ is in
  class $\clH{p,q}$.
\end{proof}

\begin{remk}
  One can also classify a local ring $R$ of codepth $3$ based on the
  invariants $e$, $h$, $l$, $n$, $\b_2, \ldots \b_5$, and
  $\mu_{e-2}$. In the case $p=3$ one then discriminates between the
  classes by looking at $\b_5$, which is $\b_4 + l\b_3 +(n-3)\b_2 +
  \tau$ with $\tau = 0$ if $R$ is in class $\clH{3,q}$ and $\tau = 1$
  if $R$ is in class $\clT$. However, it is not possible to classify
  $R$ based on Betti numbers alone. Indeed, rings in the classes
  $\clB$ and $\clH{1,1}$ have identical Poincar\'e series and so do
  rings in the classes $\clG{r}$ and $\clH{0,1}$.
\end{remk}

\begin{remk}
  \label{rmk:chn}
  A local ring $R$ of codepth $c \le 2$ can be classified based on the
  invariants $c$, $h$, and $n$. Indeed, if $c \le 1$ then $R$ is a
  hypersurface; i.e.\ it belongs to class $\clC{c}$. If $c=2$ then $R$
  belongs to class $\clC{2}$ if and only if it is Gorenstein ($h=0$
  and $n=1$); otherwise it belongs to class $\clS$.
\end{remk}

\begin{algo}
  \label{alg}
  From \rmkref{chn} and the proof of \prpref{pqr} one gets the
  following algorithm that takes as input invariants of a local ring
  of codepth $c\le 3$ and outputs its~class.
  \begin{list}{$\scriptstyle\blacksquare$}{
      \setlength{\leftmargin}{4em} \setlength{\labelwidth}{2em}
      \setlength{\itemsep}{.5ex}}
  \item[INPUT:] $c$, $e$, $h$, $l$, $n$, $\b_2$, $\b_3$, $\b_4$,
    $\mu_{e-2}$, $\mu_{e-1}$
  \item In case $c \le 1$ set \emph{Class} $= \clC{c}$
  \item In case $c=2$
    \begin{list}{$\diamond$}{ \setlength{\topsep}{0mm} }
    \item if (\ $h = 0$ and $n=1$\ ) then set \emph{Class} $ =
      \clC{2}$
    \item else set \emph{Class} $= \clS$
    \end{list}
  \item In case $c = 3$
    \begin{list}{$\diamond$}{ \setlength{\topsep}{0mm} }
    \item if (\ $h=0$ and $n=1$\ ) then set $r=l+1$
      \begin{list}{$\bullet$}{ \setlength{\topsep}{0mm} }
      \item if $r=3$ then set \emph{Class} $ = \clC{3}$
      \item else set \emph{Class} $ = \clG{r}$
      \end{list}
    \item else compute $p$ and $q$
      \begin{list}{$\bullet$}{ \setlength{\topsep}{0mm} }
      \item if (\ $q \ge 2$ or $p=2$ or $p \ge 4$\ ) then set
        \emph{Class} $ = \clH{p,q}$
      \item else compute $r$
        \begin{list}{$\circ$}{ \setlength{\topsep}{0mm} }
        \item In case $p=0$
          \begin{list}{$-$}{ \setlength{\topsep}{0mm} }
          \item if $q=r$ then set \emph{Class} $ = \clH{0,q}$
          \item else set \emph{Class} $ = \clG{r}$
          \end{list}
        \item In case $p=1$
          \begin{list}{$-$}{ \setlength{\topsep}{0mm} }
          \item if $q=r$ then set \emph{Class} $ = \clH{1,q}$
          \item else set \emph{Class} $ = \clB$
          \end{list}
        \item In case $p=3$
          \begin{list}{$-$}{ \setlength{\topsep}{0mm} }
          \item if $\mu_{e-1} = \mu_{e-2} + ln - 2$ then set
            \emph{Class} $ = \clT$
          \item else set \emph{Class} $ = \clH{3,q}$
          \end{list}
        \end{list}
      \end{list}
    \end{list}
  \item[OUTPUT:] \emph{Class}
  \end{list}
\end{algo}

\begin{remk}
  \label{rmk:mu}
  Given a local ring $R=Q/I$ the invariants $e$ and $h$ can be
  computed from $R$, and $c$, $l$, and $n$ can be determined by
  computing a minimal free resolution of $R$ over $Q$. The Betti
  numbers $\b_2,\b_3,\b_4$ one can get by computing the first five
  steps of a minimal free resolution $F$ of $k$ over $R$. Recall the
  relation $d = e-c$; the Bass numbers $\mu_{e-2}$ and $\mu_{e-1}$ one
  can get by computing the cohomology in degrees $d+1$ and $d+2$ of
  the dual complex $F^* = \Hom{F}{R}$. For large values of $d$, this
  may not be feasible, but one can reduce $R$ modulo a regular
  sequence $\mathbf{x} = x_1,\ldots,x_d$ and obtain the Bass numbers
  as $\mu_{d+i}(R) = \mu_i(R/(\mathbf{x}))$;
  cf.~\cite[3.1.16]{bruher}.
\end{remk}

\section{The implementation}
\label{sec:imp}

The \emph{Macaulay\,2} package \emph{CodepthThree} implements
Algorithm~\ref{alg}.  The function \emph{torAlgClass} takes as input a
quotient $Q/I$ of a polynomial algebra, where $I$ is contained in the
irrelevant maximal ideal $\mathfrak{N}$ of $Q$. It returns the class
of the local ring $R$ obtained by localization of $Q/I$ at
$\mathfrak{N}$.  For example, the local ring obtained by localizing
the quotient
\begin{equation*}
  \poly[\QQ]{x,y,z}/(xy^2,xyz,yz^2,x^4-y^3z,xz^3-y^4)
\end{equation*}
is in class $\clG{2}$; see \cite{LWCOVla}. Here is how it looks when
one calls the function \emph{torAlgClass}.
\begin{verbatim}
Macaulay2, version 1.6
with packages: ConwayPolynomials, Elimination, IntegralClosure, 
LLLBases, PrimaryDecomposition, ReesAlgebra, TangentCone

i1 : needsPackage "CodepthThree";
i2 : Q = QQ[x,y,z];
i3 : I = ideal (x*y^2,x*y*z,y*z^2,x^4-y^3*z,x*z^3-y^4);
o3 : Ideal of Q
i4 : torAlgClass (Q/I)
o4 = G(2)
\end{verbatim}

Underlying \emph{torAlgClass} is the workhorse function
\emph{torAlgData} which returns a hash table with the following data:

\vspace{.5\baselineskip}
\begin{tabular}[c]{l|l}
  Key & Value\\
  \hline
  \texttt{"c"} & codepth of $R$\\
  \texttt{"e"} & embedding dimension of $R$\\
  \texttt{"h"} & Cohen--Macaulay defect of $R$\\
  \texttt{"m"} & minimal number of generators of defining ideal of $R$\\
  \texttt{"n"} & type of $R$\\
  \texttt{"Class"} & (non-parametrized) class of $R$\\
  & (`B', `C', `G', `H', `S', `T', `codepth $> 3$', or `zero ring')\\
  \texttt{"p"} & rank of $A_1\cdot A_1$ \\
  \texttt{"q"} & rank of $A_1\cdot A_2$\\
  \texttt{"r"} & rank of $\delta\colon A_2 \to \Hom[k]{A_1}{A_3}$\\
  \texttt{"PoincareSeries"} & Poincar\'e series of $R$\\
  \texttt{"BassSeries"} & Bass series of $R$\\
\end{tabular}

\vspace{.5\baselineskip}
\noindent In the example from above one gets:
\begin{verbatim}
i5 : torAlgData(Q/I)
                                        2    3    4
                              2 + 2T - T  - T  + T
o5 = HashTable{BassSeries => ----------------------  }
                                       2     3    4
                             1 - T - 4T  - 2T  + T
               c => 3
               Class => G
               e => 3
               h => 1
               m => 5
               n => 2
               p => 0
                                               2
                                        (1 + T)
               PoincareSeries => ----------------------
                                           2     3    4
                                 1 - T - 4T  - 2T  + T
               q => 1
               r => 2
\end{verbatim}

To facilitate extraction of data from the hash table, the package
offers two functions \emph{torAlgDataList} and \emph{torAlgDataPrint}
that take as input a quotient ring and a list of keys.  In the example
from above one gets:
\begin{verbatim}
i6 : torAlgDataList( Q/I, {"c", "Class", "p", "q", "r", "PoincareSeries"} )

                                   2
                            (1 + T)
o6 = {3, G, 0, 1, 2, ----------------------}
                               2     3    4
                     1 - T - 4T  - 2T  + T

o6 : List

i7 : torAlgDataPrint( Q/I, {"e", "h", "m", "n", "r"} )

o7 = e=3 h=1 m=5 n=2 r=2 
\end{verbatim}

As discussed in \rmkref{mu}, the computation of Bass numbers may
require a reduction modulo a regular sequence. In our implementation
such a reduction is attempted if the embedding dimension of the local
ring $R$ is more than $3$. The procedure involves random choices of
ring elements, and hence it may fail. By default, up to 625 attempts
are made, and with the function \emph{setAttemptsAtGenericReduction,}
one can change the number of attempts. If none of the attempts are
successful, then an error message is displayed:
\begin{verbatim}
i8 : Q = ZZ/2[u,v,w,x,y,z];

i9 : R = Q/ideal(x*y^2,x*y*z,y*z^2,x^4-y^3*z,x*z^3-y^4);

i10 : setAttemptsAtGenericReduction(R,1)

o10 = 1 attempt(s) will be made to compute the Bass numbers via a generic
      reduction

i11 : torAlgClass R
stdio:11:1:(3): error: Failed to compute Bass numbers. You may raise the 
	    number of attempts to compute Bass numbers via a generic reduction 
	    with the function setAttemptsAtGenericReduction and try again.

i12 : setAttemptsAtGenericReduction(R,25)

o12 = 625 attempt(s) will be made to compute the Bass numbers via a generic
      reduction

i13 : torAlgClass R

o13 = G(2)
\end{verbatim}
Notice that the maximal number of attempts is $n^2$ where $n$ is the
value set with the function \emph{setAttemptsAtGenericReduction}.

\paragraph*{Notes.}
Given $Q/I$ our implementation of Algorithm~\ref{alg} in
\emph{torAlgData} proceeds as follows.
\begin{enumerate}
\item Check if a value is set for \emph{attemptsAtBassNumbers}; if not
  use the default value $25$.

\item Initialize the invariants of $R$ (the localization of $Q/I$ at
  the irrelevant maximal ideal) that are to be returned; see the table
  in \secref{imp}.

\item Handle the special case where the defining ideal $I$ or $Q/I$ is
  $0$. In all other cases compute the invariants $c$, $e$, $h$, $m\;(=
  l+1)$, and $n$.

\item If possible, classify $R$ based on $c$, $e$, $h$, $m$, and
  $n$. At this point the implementation deviates slightly from
  Algorithm~\ref{alg}, as it uses that all rings with $c=3$ and $h=2$
  are of class $\clH{0,0}$; see \cite[3.5]{LLA12}.

\item For rings not classified in step 3 or 4 one has $c=3$;
  cf.~\rmkref{chn}. Compute the Betti numbers $\b_2$, $\b_3$, and
  $\b_4$, and with the formula from \lemref{pqr} compute $p$ and
  $q$. If possible classify $R$ based on these two invariants.

\item For rings not classified in steps 3--5, compute the Bass numbers
  $\mu_{e-2}$ and $\mu_{e-1}$. If $d = e-3$ is positive, then the Bass
  numbers are computed via a reduction modulo a regular sequence of
  length $d$ as discussed above. Now compute $r$ with the formula
  from \lemref{pqr} and classify $R$.

\item The class of $R$ together with the invariants $c$, $l = m-1$,
  and $n$ determine its Bass and Poincar\'e series;
  cf.~\cite[2.1]{LLA12}.

\end{enumerate}

If $I$ is homogeneous, then various invariants of $R$ can be
determined directly from the graded ring $Q/I$. If $I$ is not
homogeneous, and $R$ hence not graded, then functions from the package
\emph{LocalRings} are used.

\bibliographystyle{plain}

\providecommand{\MR}[1]{\mbox{\href{http://www.ams.org/mathscinet-getitem?mr=#1}{#1}}}
\renewcommand{\MR}[1]{\mbox{\href{http://www.ams.org/mathscinet-getitem?mr=#1}{#1}}}


\end{document}